\documentclass[peerreview,twoside,web]{ieeecolor}
\usepackage{generic}
\usepackage{cite}
\usepackage{amsmath,amssymb,amsfonts}
\usepackage{algorithmic}
\usepackage{graphicx}
\usepackage{algorithm,algorithmic}
\usepackage{hyperref}
\hypersetup{hidelinks=true}
\usepackage{textcomp}
\usepackage{subcaption} 
\usepackage{amsmath}
\usepackage{amsfonts}
\usepackage{graphicx}
\newtheorem{lemma}{Lemma}
\newtheorem{theorem}{Theorem}
\newtheorem{corollary}{Corollary}
\newtheorem{definition}{Definition}

\newtheorem{example}{Example}
\newtheorem{problem}{Problem}
\newtheorem{assumption}{Assumption}
\newtheorem{proposition}{Proposition}
\newcommand{\defeq}{\overset{\mathrm{def}}{=}}

\def\BibTeX{{\rm B\kern-.05em{\sc i\kern-.025em b}\kern-.08em
    T\kern-.1667em\lower.7ex\hbox{E}\kern-.125emX}}
\begin{document}
\setcounter{page}{1}
\title{Fundamental limitations of monotonic tracking systems}
\author{Hamed Taghavian
\thanks{``This work was supported by KTH Royal Institute of Technology.'' }
\thanks{Hamed Taghavian is with the Department of Intelligent Systems, KTH Royal Institute of Technology, 10044 Stockholm, Sweden (e-mail: hamedta@kth.se).}}

\maketitle

\begin{abstract}
We consider the monotonic tracking control problem for continuous-time single-input single-output linear systems using output-feedback linear controllers in this paper. We provide the necessary and sufficient conditions for this problem to be solvable and expose its fundamental limitations: the exact feasible locations of the plant zeros, the minimum controller order possible, and the fastest decay rate achievable for the closed-loop system. The relationship between these bounds is explained by a simple geometric shape for plants with a pair of complex-conjugate zeros.



\end{abstract}

\begin{IEEEkeywords}
Linear control systems, Monotonic tracking, Output-feedback, Overshoots, Undershoots. 
\end{IEEEkeywords}

\section{Introduction}\label{sec:introduction}
\IEEEPARstart{I}nherent limitations of control systems were studied by Bode in 1945 and were further investigated by several authors in the 80s~\cite[\S3]{thebook}. More recently, this line of research has led to several performance bounds that determine what is achievable by a control system and what is not, with regard to noise suppression, disturbance rejection, sensitivity, and transients~\cite{thebook,Diabet,zerocrossing,onovershoot,tac2019,tac2024,TAC,MTNS}. A classical example is Bode's integral theorem, suggesting that sensitivity attenuation cannot be achieved at \emph{all} frequencies, regardless of the controller design. These results provide a benchmark for comparing different controllers and save the community from making futile attempts to beat the inescapable limits of performance.   


\subsection{Monotonic tracking systems}
We study the limits of performance in monotonic tracking control systems. These systems track reference set points monotonically, eliminating both overshoots and undershoots in the closed-loop system response~\cite{peak}. Monotonic tracking is especially important when the optimal set point is close to an economic or safety constraint in a control system \cite{PhS1988}. A classical example is controlling a robotic manipulator whose optimal trajectory is close to obstacles~\cite{DeV1990}. Monotonic tracking systems are used in numerous applications, such as robotics \cite{HST2012}, process control \cite{DeV1990}, aerospace \cite{XBS2018}, and transportation~\cite{ScL2021}.  

Several techniques are available for designing monotonic tracking controllers. For example, state-feedback controllers are used in \cite{ZhZ2019,NTSF2015,DaJ2022} to ensure a monotonic closed-loop system response. Another line of research proposes output-feedback controllers for monotonic tracking~\cite{cdc,TaJ2020,combina,ECC,ScN2014}. 

Although these studies provide useful frameworks for designing monotonic tracking controllers, most of them do not discuss the existence of a solution. Therefore, when one of these methods fails, it is not clear whether the design method is too conservative or if the problem is infeasible in the first place. This emphasizes the importance of studying the fundamental limitations in monotonic tracking systems to reveal the feasibility conditions in these problems. Although such feasibility conditions are derived for state-feedback controllers in \cite{NPG2019} and for output-feedback controllers in \cite{darbha2003}, several open questions remain to be answered~\cite{ZhW2015}. In particular, the relationship between the plant characteristics, the complexity of the controller, and different closed-loop system properties, and how they limit each other in a monotonic tracking system, is still unknown. In this paper, we obtain more general feasibility conditions for monotonic tracking systems using output-feedback controllers that will clarify these relations.

\subsection{The role of non-minimum-phase zeros}
The most general linear time-invariant output-feedback control structure is realized by the following two-parameter controller~\cite[\S5.6]{vidbook}
\begin{equation}\label{eqn:two-par}
    U(s) = \frac{N(s)}{D(s)} R(s) - \frac{F(s)}{G(s)} Y(s),  
\end{equation}
where $U$, $R$, and $Y$ are the plant's input, reference input, and the plant's output, respectively. Therefore, a design objective is achievable by linear controllers if and only if it can be realized by a two-parameter controller as (\ref{eqn:two-par}). This controller can result in any stable closed-loop system as long as its transfer function contains the non-minimum-phase zeros of the plant~\cite[p.106]{vidbook}. Therefore, the non-minimum-phase zeros of the plant pose a fundamental limitation on general linear control system design. These limitations have been acknowledged in different contexts, such as in H-$\infty$ control, where the non-minimum-phase zeros are responsible for the ``push-pop'' effect in the sensitivity function~\cite{zames1984h}.\looseness=-1

In monotonic tracking control, the link between real non-minimum-phase zeros and undershoots in the closed-loop system response is well documented~\cite{HoB2007}. It was first shown in \cite{Vid1986} that it is impossible to obtain a monotonic response if the system has an odd number of real right-half plane zeros, and it was later shown that this effect is present with both odd and even numbers of positive zeros~\cite{Sko2005}. Finally, in \cite{darbha2003}, it was proved that a two-parameter controller (\ref{eqn:two-par}) exists that achieves monotonic tracking if and only if the plant does not have any real non-negative zeros.\looseness=-1



\subsection{Contributions}
In this paper, we show that considering a special family of two-parameter controllers (\ref{eqn:two-par}) with
\begin{align}\label{simple2par}
    N(s)&=K_c \nonumber\\
    D(s)&=G(s) 
\end{align}
is enough to achieve monotonic tracking, where $K_c$ is a static gain. Hence, monotonic tracking is feasible if and only if it is feasible using this simplified class of controllers. In addition, we introduce a transformation of plant zeros that provides a deeper understanding of the fundamental limitations present in monotonic tracking systems than the original plant zeros (Theorem~\ref{thm:feasibility}). In particular, by using this transformation, we determine how fast the closed-loop system can be designed, how low the controller order can be chosen, and how close the plant zeros can get to the positive real axis for the monotonic tracking control problem to be feasible with a given controller order and a closed-loop decay rate not slower than a given value.


Note that the controller order is allowed to get arbitrarily large in \cite{darbha2003} to achieve monotonic tracking. In this paper, we provide the necessary and sufficient conditions of monotonic tracking for any \emph{given} controller order. In addition, we obtain the fastest decay rate achievable for the closed-loop system with a monotonic response. The majority of available studies only focus on the role of \emph{real} zeros on these performance limitations~\cite{zerocrossing,onovershoot}. However, our results can be applied to plants with zeros located anywhere in the complex plane.



\subsection{Organization}
This paper is structured as follows. In Section~\ref{sec:problem_statement} we propose the problem, motivate the assumptions, and derive the preliminary results. 
In Section~\ref{sec:feasibility}, the feasibility of the monotonic tracking control problem is addressed. In Section~\ref{sec:fundamental_limits}, the fundamental limitations of a monotonic tracking system are studied and demonstrated in several examples. Conclusion remarks and a summary of key results are provided in Section~\ref{sec:conclusion}. 

\subsection{Notation}\label{sec:notation}
We use he following notation in this paper: A real vector $z\in\mathbb{R}^m$ sorted in descending order of its elements is denoted by $z^{\downarrow}$. The complex-conjugate of $z\in\mathbb{C}$ is $\bar{z}$, the inverse Laplace transform is denoted by $\mathfrak{L}^{-1}\lbrace .\rbrace$ and $\delta(t)$ denotes the Dirac delta function. For a polynomial $Q(x)$, the coefficient corresponding to $x^{i}$ is denoted by $\left[Q(x)\right]_i$ and we write $Y\succeq 0$ to show that matrix $Y$ is positive semi-definite.

\section{Monotonic tracking control}\label{sec:problem_statement}
\subsection{Problem statement}


Consider the single-input single-output linear time-invariant continuous-time system with the $n_{\rm o}$th order transfer function
\begin{align}\label{eqn:openloop_transfer_function}
    H_{\rm o}(s)&=\frac{B_{\rm o}(s)}{A_{\rm o}(s)}=\frac{K_{\rm o} \prod_{i=1}^{m} (s-z_i)}{\prod_{i=1}^{n_{\rm o}} (s-p^{\rm o}_i)}\nonumber\\
    &=\frac{ b^{\rm o}_0 s^{n_{\rm o}}+b^{\rm o}_1 s^{n_{\rm o}-1}+ \cdots +b^{\rm o}_{n_{\rm o}}}{ s^{n_{\rm o}}+a^{\rm o}_1 s^{n_{\rm o}-1}+ \cdots +a^{\rm o}_{n_{\rm o}}}
\end{align}
connected with the two-parameter controller
\begin{equation}\label{eqn:two-par-restricted}
    U(s) = \frac{K_c}{G(s)} R(s) - \frac{F(s)}{G(s)} Y(s), 
\end{equation}
where
\begin{equation}\label{eqn:F,G}
F(s)=\sum_{k=0}^{n_c} f_k s^{n_c-k},\;G(s)=\sum_{k=0}^{n_c} g_k s^{n_c-k}    
\end{equation}
and $K_c \in \mathbb{R}$ is the static gain of the controller. The closed-loop system has the following transfer function~\cite{ScL2020_IFAC}
\begin{align}\label{eqn:transfer_function}
H(s)&=\frac{B(s)}{A(s)}=\frac{K_c B_{\rm o}(s)}{B_{\rm o}(s)F(s)+A_{\rm o}(s)G(s)}\nonumber\\
&=\frac{ b_0 s^{n}+b_1 s^{n-1}+ \cdots +b_{n}}{ s^{n}+a_1 s^{n-1}+ \cdots +a_{n}}=K \frac{\prod_{i=1}^{m} (s-z_i)}{\prod_{i=1}^{n} (s-p_i)},
\end{align}
where $b_k=0$ for $0\leq k\leq n-m-1$, and 
\begin{align*}
K&=K_c K_{\rm o},   \\ 
n&=n_c+n_{\rm o}.
\end{align*}
The controller (\ref{eqn:two-par-restricted}) is a special case of the two-parameter controller (\ref{eqn:two-par}) in which (\ref{simple2par}) holds. 
This controller decouples the closed-loop zeros from the closed-loop poles, where the former is fixed and equal to the plant zeros. This independence of parameters in the numerator and denominator significantly simplifies the control synthesis for this structure compared to more general two-parameter controllers (\ref{eqn:two-par}). Hence, the simplified two-parameter controller (\ref{eqn:two-par-restricted}) has been used to ensure the transient performance of closed-loop systems in several instances, including monotonic tracking control problems~\cite{onovershoot,LCM,cdc}.

This paper aims to study the \emph{feasibility} of the following problem.

\begin{problem}[Monotonic tracking control problem]\label{prob:MTCP}
    Given the plant (\ref{eqn:openloop_transfer_function}), find the polynomials (\ref{eqn:F,G}) and the static gain $K_c \in \mathbb{R}$ of the controller (\ref{eqn:two-par-restricted}) such that the closed-loop system with transfer function~(\ref{eqn:transfer_function}):
    \begin{itemize}
        \item [(i)] is stable,
        \item [(ii)] has a monotonic step response,
        \item [(iii)] has a unit steady-state gain. 
    \end{itemize}
\end{problem}

Our goal is to identify conditions for the existence of a controller (\ref{eqn:two-par-restricted}) that can solve Problem~\ref{prob:MTCP}. One such condition is given in the following proposition.

\begin{proposition}\label{prop:z>=0}
Problem~\ref{prob:MTCP} does \emph{not} have a solution if the plant (\ref{eqn:openloop_transfer_function}) has real non-negative zeros.
\end{proposition}
\begin{proof}
This result is a special case of a known property of linear systems, which dates back to \cite{widder1934}: If the impulse response of a linear system changes sign $k$ times, then its transfer function can have at most $k$ real zeros which are located on the right side of all the poles \cite{darbha2003}. Therefore, if the closed-loop transfer function (\ref{eqn:transfer_function}) is stable and has a monotonic step response (a non-negative impulse response), it can at most have $k=0$ non-negative zeros. However, the zeros of the closed-loop system (\ref{eqn:transfer_function}) coincide with the zeros of the plant (\ref{eqn:openloop_transfer_function}). Therefore, if the monotonic tracking control problem has a solution, then the plant (\ref{eqn:openloop_transfer_function}) does not have any real non-negative zeros.
\end{proof}

Proposition~\ref{prop:z>=0} provides a \emph{necessary} condition for the feasibility of Problem~\ref{prob:MTCP}. This condition was also shown to be sufficient for monotonic tracking using \emph{general} two-parameter controllers (\ref{eqn:two-par}) in \cite{darbha2003}. We will show that this condition is also sufficient for the simplified controllers (\ref{eqn:two-par-restricted}) in the next sections. Nevertheless, the condition provided in Proposition~\ref{prop:z>=0} on its own does not provide any insights into the fundamental limitations of monotonic tracking, that is, the feasible range of the key parameters, such as the controller order, the decay rate of the closed-loop system, etc., when solutions exist. Therefore, we obtain a more nuanced condition that generalizes Proposition~\ref{prop:z>=0} and gives new insights into these performance bounds in the next sections (see Theorem~\ref{thm:feasibility}).
 


We make the following assumption throughout the paper to study the feasibility of Problem~\ref{prob:MTCP} in the non-trivial cases.
\begin{assumption}[Plant assumptions]\label{ass:plant}
The plant (\ref{eqn:openloop_transfer_function}) satisfies: 
    \begin{itemize}
    \item[1)] $z_i\not\in [0,+\infty)$ for all $i=1,2,\dots,m$.
    \item[2)] $K_{\rm o}\neq 0$.
\end{itemize}
\end{assumption}
When the first assumption is not met, Ptoblem~\ref{prob:MTCP} is not feasible, according to Proposition~\ref{prop:z>=0}. The second assumption removes the degenerate case $H_{o}(s)\equiv 0$.

\subsection{Problem objectives}
To see how the objectives (i)-(iii) affect the feasibility of Problem~\ref{prob:MTCP}, we study each in a separate section.

\subsubsection{Stability}
A solution to Problem~\ref{prob:MTCP} renders the closed-loop system stable.
The closed-loop system (\ref{eqn:transfer_function}) is stable if and only if the roots of the closed-loop characteristic equation
$$
A(s)=B_{\rm o}(s)F(s)+A_{\rm o}(s)G(s)=0
$$
have negative real parts. This condition also ensures \emph{internal} stability in the control structure in Figure~\ref{fig:control_scheme} when $G(s)$ is stable. Otherwise, when $G(s)$ is unstable, the same condition ensures internal stability in the alternative control structure in Figure~\ref{fig:alt_control_scheme}. Both structures in Figures~\ref{fig:control_scheme} and \ref{fig:alt_control_scheme} are equivalent and implement the controller~(\ref{eqn:two-par-restricted})~\cite{LCM}. 

\begin{figure}[!t]
\centerline{\includegraphics[width=0.5\columnwidth]{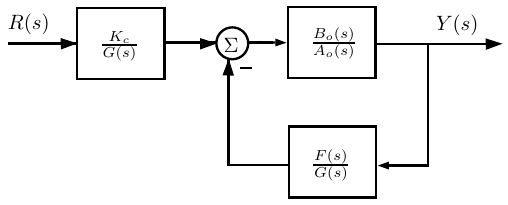}}
\caption{Output feedback control scheme.}
\label{fig:control_scheme}
\end{figure}
\begin{figure}[!t]
\centerline{\includegraphics[width=0.5\columnwidth]{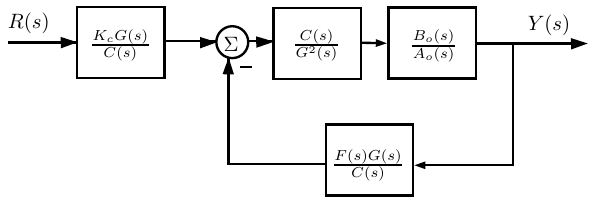}}
\caption{Output feedback control scheme (alternative implementation).}
\label{fig:alt_control_scheme}
\end{figure}

\subsubsection{Monotonic step response}
A solution to Problem~\ref{prob:MTCP} ensures the closed-loop system $H(s)$ has a monotonically increasing step response. For a linear system, this condition is equivalent to having a non-negative impulse response, \emph{i.e.}, being an externally positive system. These systems are defined as follows.

\begin{definition}[\cite{BSFG2017}]\label{def:PIR}
The transfer function $H:\mathbb{C} \to\mathbb{C}$ is externally positive if its inverse Laplace transform satisfies $h(t)=\mathfrak{L}^{-1}\lbrace H(s) \rbrace \geq 0$ for all $t \in [0,+\infty)$.
\end{definition}

A simple necessary condition for a transfer function to be externally positive is to have a non-negative static gain, according to the following proposition.

\begin{proposition}\label{prop:PIRnecK>0}
Let $K\neq 0$. A necessary condition for system (\ref{eqn:transfer_function}) to be externally positive, is
\begin{equation}\label{eqn:Kcl>0}
K=K_cK_{\rm o} > 0.
\end{equation}
\end{proposition}
\begin{proof}
See Appendix~\ref{sec:Proof_of_prop:PIRnecK>0}.
\end{proof}

The condition presented in Proposition~\ref{prop:PIRnecK>0} is only necessary. Finding conditions that are both necessary and sufficient for a transfer function to be externally positive is not trivial. Characterizing the set of all externally positive systems in the frequency domain or among state-space models is an open problem~\cite{LCM}. These characterizations are only available for first- and second-order systems, which are presented in the next propositions.

\begin{proposition}[\cite{LCM,ScL2020_IFAC}]\label{prop:ex_pos_n=1}
The first-order transfer function (\ref{eqn:transfer_function}) is externally positive if and only if $K>0$ and $B(p^{{\rm cl}\downarrow}_1)\geq 0$.
\end{proposition}

\begin{proposition}[\cite{LCM}]\label{prop:ex_pos_n=2}
The second-order transfer function (\ref{eqn:transfer_function}) is externally positive if and only if $p_1,p_2\in\mathbb{R}$, $K>0$, $B(p^{\downarrow}_1),\frac{d}{ds}B(p^{\downarrow}_1)\geq 0$ and $B(p^{\downarrow}_1)\geq B(p^{\downarrow}_2)$, assuming there are no zero-pole cancellations in (\ref{eqn:transfer_function}).
\end{proposition}

Unfortunately, such general necessary and sufficient conditions do not exist for higher-order systems. Hence, when the system order is not restricted, one needs to impose restrictions on other system parameters to obtain the necessary and sufficient conditions that ensure the system is externally positive. The following lemma does so, by assuming all the poles are real and equal.

\begin{lemma}\label{lem:nec&suff}
Let $n\geq 1$ and $
p_1=p_2=\dots=p_{n}=\sigma
$. System (\ref{eqn:transfer_function}) is externally positive if and only if $K> 0$ and
\begin{equation}\label{eqn:Q(x)}
    Q(\sigma,t)\defeq \sum_{j=0}^{m}\binom{n-1}{j}B^{(j)}(\sigma)t^{m-j}\geq 0
\end{equation}
holds on $t\in(0,+\infty)$.
\end{lemma}
\begin{proof}
See Appendix~\ref{sec:Proof_of_lem:nec&suff}.
\end{proof}

Lemma~\ref{lem:nec&suff} reduces the problem of verifying a linear system as externally positive to verifying the non-negativity of a univariate polynomial.

\subsubsection{Unit steady-state gain}
A solution to Problem~\ref{prob:MTCP} ensures the closed-loop system (\ref{eqn:transfer_function}) has a unit steady-state gain. This property is required for the closed-loop system to track a step input with zero steady-state error. A unit steady-state gain is achieved by choosing the static gain of the controller $K_c$ as~\cite{LCM}
\begin{equation}\label{eqn:KcDesign}
K_c=\left(B_{\rm o}(0)F(0)+A_{\rm o}(0)G(0)\right)/B_{\rm o}(0),
\end{equation}
according to the final value theorem. The plant must satisfy $B_{\rm o}(0)\neq 0$ for (\ref{eqn:KcDesign}) to exist. This condition is always met by Assumption~\ref{ass:plant}.


\subsection{An equivalent problem}

In this section, we transform Problem~\ref{prob:MTCP} into an equivalent problem whose feasibility is more tractable to analyze. For this transformation to be valid, we need the following set of assumptions throughout the paper.

\begin{assumption}[Controller assumptions]\label{ass:controller}
The controller (\ref{eqn:two-par-restricted}) satisfies: 
\begin{itemize}
    \item[1)] $n_c \geq n_{\rm o}-1$.
    \item[3)] $n_c>m-n_{\rm o}$. 
    \item[2)] $K=K_c K_{\rm o}> 0$.
\end{itemize}
\end{assumption}
The first item ensures the controller has enough freedom to place the closed-loop poles anywhere on the complex plane. The second item in Assumption~\ref{ass:controller} ensures the closed-loop system is strictly proper, \emph{i.e.}, $n>m$. We make this assumption because Problem~\ref{prob:MTCP} is always feasible otherwise, \emph{i.e.}, when $n=m$. To see this, we note that under the assumption $n_c \geq n_{\rm o}-1$, the equality $n=m$ is only possible in the following two cases:

\noindent\textit{Case 1: $n_c=n_{\rm o}=m=0$}. In this case, $A(s)=K_{\rm o}f_0+g_0=1$ and the closed-loop system (\ref{eqn:transfer_function}) is given by $H(s)=K_cK_{\rm o}=K$. Thus, choosing $K_c=1/K_{\rm o}$ solves the monotonic tracking problem.

\noindent\textit{Case 2: $n_c=0,n_{\rm o}=m=1$}. In this case, the closed-loop system is given by
$$
H(s)=\frac{K_cK_o(s-z_1)}{f_0K_o(s-z_1)+g_0(s-p^{\rm o}_1)}=K\frac{s-z_1}{s-p_1},
$$
where
$$
\begin{array}{l}
1=K_{\rm o}f_0+g_0, \\
p_1= z_1f_0K_{\rm o}+p^{\rm o}_1g_0. 
\end{array}
$$
To make sure $H(s)$ is stable and has a monotonically increasing step response it is required that~\cite{ScL2020_IFAC}
$$
z_1<p_1<0.
$$
According to Assumption~\ref{ass:plant} we have $z_1<0$. Hence, the monotonic tracking problem is solved by choosing $g_0$, $f_0$ and $K_c$ such that
\begin{align*}
    &g_0(p^{\rm o}_1-z_1)\in(0,-z_1), \\
    &f_0=(1-g_0)/K_{\rm o},\\
    &K_c=p_1/(z_1K_{\rm o}).
\end{align*}
Therefore, the second item in Assumption~\ref{ass:controller} is not restrictive. The last item in Assumption~\ref{ass:controller}, is not restrictive either because one may always choose the sign of the controller gain $K_c$ such that $K=K_cK_{\rm o}>0$. This inequality is satisfied for all the solutions to problem~\ref{prob:MTCP} when they exist. To see this, note that any solution to Problem~\ref{prob:MTCP} renders the closed-loop system externally positive, which entails $K\geq 0$ according to Proposition~\ref{prop:PIRnecK>0}. Since the closed-loop system is also stable with a unit steady-state gain, $K_c\neq 0$ holds in (\ref{eqn:KcDesign}). As $K_{\rm o}\neq 0$ holds by Assumption~\ref{ass:plant}, one has $K=K_cK_{\rm o}\neq 0$.



Now, consider the following problem.

\begin{problem}\label{prob:monotonic_stable}
Given the plant zeros $z\in\mathbb{C}^m$, find the closed-loop poles $p\in\mathbb{C}^{n}$ such that the closed-loop system $H(s)$ is stable and externally positive. 
\end{problem}

As the next lemma shows, the above problem is equivalent to Problem~\ref{prob:MTCP}.


\begin{lemma}\label{lem:equivalence}
Problem~\ref{prob:MTCP} is feasible, if and only if Problem~\ref{prob:monotonic_stable} is feasible. 
\end{lemma}
\begin{proof}
\textit{if part:}
Assume that Problem~\ref{prob:monotonic_stable} is feasible, \emph{i.e.}, there is some $p\in\mathbb{C}^{n}$ such that (\ref{eqn:transfer_function}) is stable and externally positive. Then, by following a standard pole-placement procedure, one can determine the controller polynomials (\ref{eqn:F,G}) that realize the closed-loop system (\ref{eqn:transfer_function}), by solving the linear algebraic equation (see \cite{LCM})
\begin{equation}\label{eqn:M[f,g]=a^cl}
M\begin{bmatrix} f\\g\end{bmatrix}=a,
\end{equation}
where the elements of $M\in\mathbb{R}^{(n_{\rm o}+n_c+1)\times 2(n_c+1)}$ are defined in (\ref{eqn:elements_of_M})
\begin{figure*}[!t]
\begin{equation}\label{eqn:elements_of_M}
[M]_{ij}=\left\lbrace
\begin{array}{lll}
b^{\rm o}_{i-j}, & 1\leq j\leq n_c+1 & j\leq i\leq j+n_{\rm o} \\
a^{\rm o}_{i-j+n_c+1}, & n_c+2\leq j\leq 2n_c+2 & j-n_c-1\leq i\leq j-n_c-1+n_{\rm o} \\
0, & \textnormal{otherwise} &
\end{array}\right.
\end{equation}
\end{figure*}
and $a\in\mathbb{R}^{n_{\rm o}+n_c+1}$ consists of the desired characteristic equation coefficients given by the identity
\begin{equation}\label{eqn:acl,pcl}
\sum_{k=0}^{n}a_k s^{n-k}=\prod_{i=1}^{n}(s-p_i).
\end{equation}
Equation (\ref{eqn:M[f,g]=a^cl}) has a unique solution when $n_c=n_{\rm o}-1$ and infinitely many solutions when $n_c>n_{\rm o}-1$, provided that $B_{\rm o}(s)$ and $A_{\rm o}(s)$ are relatively prime~\cite[p.180]{GGS:00}. Hence, Assumption~\ref{ass:controller} ensures that the above pole-placement process is well defined for all $a\in\mathbb{R}^{n+1}$, and therefore, for any set of closed-loop poles.

After the controller polynomials (\ref{eqn:F,G}) are obtained, the controller gain $K_c$ is determined as (\ref{eqn:KcDesign}) to ensure a zero stead-state tracking error. Gain (\ref{eqn:KcDesign}) always exists by Assumption~\ref{ass:plant}. Therefore, given a solution to Problem~\ref{prob:monotonic_stable}, one can always construct a solution to Problem~\ref{prob:MTCP}.

\textit{only if part:} Obvious.
\end{proof}

According to Lemma~\ref{lem:equivalence}, the monotonic tracking control problem is feasible, if and only if there is some $p\in\mathbb{C}^{n}$ in the left half plane (for stability) that makes the transfer function (\ref{eqn:transfer_function}) with a fixed set of zeros $z\in\mathbb{C}^{m}$ externally positive. As shown in the next section, a simple condition can verify whether such poles $p\in\mathbb{C}^{n}$ exist or not (Theorem~\ref{thm:feasibility}), even though the conditions on the zeros and poles that make a transfer function externally positive are generally unknown for higher-order systems~\cite{LCM}.




\section{Feasibility}\label{sec:feasibility}

As was shown in Section~\ref{sec:problem_statement}, the monotonic tracking control problem (Problem~\ref{prob:MTCP}) is feasible, if and only if Problem~\ref{prob:monotonic_stable} is feasible. Therefore, we provide the necessary and sufficient conditions that ensure Problem~\ref{prob:monotonic_stable} has a solution in this section.
To study the feasibility of Problem~\ref{prob:monotonic_stable} in the general case, we rely on two core lemmas. The first one shows a stable externally positive transfer function with real poles remains externally positive, after shifting its poles forward.

\begin{lemma}\label{lem:moving_forward}
Let $H(s)$ in (\ref{eqn:transfer_function}) be externally positive, where $p_i\in(-\infty,0)$. Then the transfer function
\begin{equation}\label{eqn:H'_1}
H'(s)= \frac{K\prod_{i=1}^{m} (s-z_i)}{\prod_{i=1}^{n} (s-p_i-\delta_i)}
\end{equation}
is also externally positive, where $\delta_i\geq 0$ for $i=1,2,\dots,n$.
\end{lemma}
\begin{proof}
See Appendix~\ref{sec:Proof_of_lem:moving_forward}.
\end{proof}

The next lemma indicates that when Problem~\ref{prob:monotonic_stable} is feasible, it can always be solved by choosing all the closed-loop poles real and equal, \emph{i.e.},
\begin{equation}\label{eqn:p1=...=pn}
    p_1=p_2=\dots=p_{n}=\sigma.
\end{equation}
Imposing the condition (\ref{eqn:p1=...=pn}) on closed-loop system poles significantly simplifies Problem~\ref{prob:monotonic_stable}, by decreasing the number of design variables from $n$ to $1$. 

\begin{lemma}\label{lem:reduction_to_real_equidistanced}
Problem~\ref{prob:monotonic_stable} is feasible if and only if there is some $\sigma<0$ such that substituting (\ref{eqn:p1=...=pn}) in (\ref{eqn:transfer_function}) makes $H(s)$ externally positive.
\end{lemma}
\begin{proof}
See Appendix~\ref{sec:Proof_of_lem:reduction_to_real_equidistanced}.
\end{proof}

According to Lemma~\ref{lem:reduction_to_real_equidistanced}, the poles in Problem~\ref{prob:monotonic_stable} can be restricted to be real and equal without loss of feasibility. Recall that all the externally positive transfer functions with real-equal poles are characterized by Lemma~\ref{lem:nec&suff}. This allows us to determine the feasibility of Problem~\ref{prob:monotonic_stable} in the general case $p\in\mathbb{C}^n$ and $z\in\mathbb{C}^m$, by using Lemmas~\ref{lem:reduction_to_real_equidistanced}~and~\ref{lem:nec&suff}.

For a convenient exposition of our result, we define the transformed numerator polynomial as follows
\begin{equation}\label{eqn:Btild}
    \tilde{B}(s):=\sum_{i=1}^n \tilde{b}_i s^{n-i},
\end{equation}
where
\begin{equation}\label{eqn:btild=b/(i-1)!}
    \tilde{b}_i=b_i/(i-1)!, \quad i=1,2,\dots,n
\end{equation}
The following theorem determines whether a set of poles $p\in\mathbb{C}^n$ exists that makes a transfer function (\ref{eqn:transfer_function}) with prescribed zeros $z\in\mathbb{C}^m$ externally positive. This result provides the necessary and sufficient condition for the feasibility of Problem~\ref{prob:monotonic_stable} (and~\ref{prob:MTCP}).

\begin{theorem}\label{thm:feasibility}
Problem~\ref{prob:monotonic_stable} is feasible if and only if the polynomial (\ref{eqn:Btild}) does \emph{not} have any real non-negative roots.
\end{theorem}
\begin{proof}
Since $K>0$ holds by Assumption~\ref{ass:controller}, the limit
$$
\lim_{s\to+\infty} \tilde{B}(s) = \lim_{s\to+\infty}\tilde{b}_{n-m}s^{m}=
\lim_{s\to+\infty}Ks^{m}/(n-m-1)!
$$
is either $+\infty$ or a positive number. Hence, polynomial $\tilde{B}(s)$ does not have any real non-negative roots if and only if
\begin{equation}\label{eqn:Btild>0}
    \tilde{B}(s)>0, \quad s\in [0,+\infty).
\end{equation}
For $s=0$, inequality (\ref{eqn:Btild>0}) is equivalent to
$b_n > 0$. Otherwise for $s\in(0,+\infty)$, we may change the variables as $t=1/s$ and write
\begin{align*}
    \tilde{B}(s)&=t^{1-n}\sum_{i=n-m}^{n}\tilde{b}_i t^{i-1} \\
    &=t^{-m}\sum_{i=0}^{m} b_{n-i}t^{m-i}/(n-i-1)! \\
    &=\frac{t^{-m}}{(n-1)!}\sum_{i=0}^{m} \binom{n-1}{i} i!b_{n-i}t^{m-i} \\
    &=\frac{t^{-m}}{(n-1)!}\sum_{i=0}^{m} \binom{n-1}{i} B^{(i)}(0) t^{m-i} \\
    &=\frac{t^{-m}}{(n-1)!}Q(0,t),
\end{align*}
which implies that condition (\ref{eqn:Btild>0}) is equivalent to
\begin{equation}\label{eqn:Q(0)>0}
b_n>0 \wedge Q(0,t)> 0,\quad t\in(0,+\infty).   
\end{equation} 
Therefore, we continue the proof by showing that Problem~\ref{prob:monotonic_stable} is feasible if and only if (\ref{eqn:Q(0)>0}) holds.

\textit{if part:}
Let (\ref{eqn:Q(0)>0}) hold. Since for every $t\in(0,\infty)$, function $Q(\sigma,t)$ is continuous in $\sigma$, there exists some neighbourhood around the origin ($\sigma=0$) with radius $\delta(t)>0$ such that
\begin{equation}\label{eqn:Qcontinuous}
\vert\sigma\vert<\delta(t) \Rightarrow \vert Q(\sigma,t)-Q(0,t)\vert<\epsilon(t),
\end{equation}
where $\epsilon(t)=Q(0,t)$. If $\delta(t)$ is bounded away from the origin, \emph{i.e.}, if there is some $\delta^{\star}>0$ such that
\begin{equation}\label{eqn:delta>deltastar}
    \delta(t)\geq \delta^{\star}, \quad
    t\in(0,+\infty)
\end{equation}
then relation (\ref{eqn:Qcontinuous}) will imply
$$
\sigma \in (-\delta^{\star},0) \Rightarrow Q(\sigma,t)>0 ,\quad t\in(0,\infty)
$$
which indicates that placing all the poles at $p_1=p_2=\dots=p_{n}=\sigma$ makes the closed-loop system stable and externally positive according to Lemma~\ref{lem:nec&suff} and thereby, Problem~\ref{prob:monotonic_stable} is feasible. Hence, it remains to prove (\ref{eqn:delta>deltastar}).
By collecting the monomials in (\ref{eqn:Q(x)}) with respect to $\sigma$, we may write
$$
Q(\sigma,t)=\sum_{i=n-m}^n q_i(t)\sigma^{n-i},
$$
where the coefficient functions $q_i(t)$ are given by
\begin{equation}\label{eqn:q_i}
    q_i(t)=\sum_{j=0}^{i-n+m}\binom{n-i+j}{j}\binom{n-1}{j}j!b_{i-j}t^{m-j}.
\end{equation}
Assume $\vert \sigma\vert<\delta(t)<1$. Then for every $t\in(0,\infty)$,
$$
\left\vert\frac{d}{d\sigma} Q(\sigma,t)\right\vert \leq
m \sum_{i=n-m}^{n-1} \vert q_i(t)\vert
$$
holds, which implies
$$
\vert Q(\sigma,t)-Q(0,t)\vert \leq
m \delta(t) \sum_{i=n-m}^{n-1} \vert q_i(t)\vert.
$$
Therefore, choosing $\delta(t)$ such that
\begin{equation}\label{eqn:deltabound}
    0<\delta(t)<\min\left\lbrace 1,
    \frac{\epsilon(t)}{m\sum_{i=n-m}^{n-1} \vert q_i(t)\vert}\right\rbrace
\end{equation}
holds for all $t\in(0,+\infty)$ will satisfy (\ref{eqn:Qcontinuous}). Next, we show that the upper bound in (\ref{eqn:deltabound}) does not vanish. Since the function $\epsilon(t)=Q(0,t)$ is a polynomial in $t$ and
$$
Q(0,0)= K(n-1)!/(n-1-m)!>0,
$$
it has a positive lower bound in the range $t\in(0,\infty)$. Also, the denominator in the right side of (\ref{eqn:deltabound}) is bounded for finite $t>0$ and when $t\to +\infty$, we have
\begin{align*}
    \lim_{t\to +\infty}
\frac{\epsilon(t)}{m \sum_{i=n-m}^n \vert q_i(t)\vert}&=\lim_{t\to +\infty}\frac{Q(0,t)}{m\sum_{i=n-m}^n \vert b_i\vert t^m}\\
&=\frac{b_n }{m\sum_{i=n-m}^n \vert b_i\vert }\neq 0.
\end{align*}
This proves that the radius $\delta(t)$ in (\ref{eqn:Qcontinuous}) can be chosen with a positive lower bound (\ref{eqn:delta>deltastar}) and concludes the sufficiency proof.

\textit{only if part:} Next, we show that Problem~\ref{prob:monotonic_stable} is not feasible when inequality (\ref{eqn:Q(0)>0}) is violated. This can happen when $b_n\leq 0$, $Q(0,t^{\star})<0$ or $Q(0,t^{\star})=0$ at some point $t=t^{\star}\in(0,+\infty)$. Each case is treated separately.

\noindent\textit{Case 1:} Assume $b_n\leq 0$. We note that $\lim_{t\to +\infty}Q(0,t)=\lim_{t\to +\infty}b_n t^m$. Therefore, if $b_n<0$, then $Q(0,t)$ becomes negative for large enough $t>0$. Otherwise, if $b_n=0$, then $B(0)=0$ holds, and assuming the closed-loop system is stable, we would have $\int_{0}^{+\infty}h(t)dt=H(0)=0$. As assumptions~\ref{ass:plant} and \ref{ass:controller} ensure $H(s)\not\equiv 0$, the closed-loop impulse response $h(t)$ has to be negative at some $t\geq 0$ and $H(s)$ can not be externally positive in this case. This shows that $b_n>0$ is necessary for feasibility of Problem~\ref{prob:monotonic_stable}.

\noindent\textit{Case 2:} Assume $Q(0,t^{\star})<0$ holds for some $t=t^{\star}\in(0,+\infty)$. We note that since function $Q(\sigma,t^{\star})$ is continuous in $\sigma$, there exists some $\delta>0$ such that
\begin{equation}\label{eqn:Qcontinuous2}
\vert\sigma\vert<\delta \Rightarrow \vert Q(\sigma,t^{\star})-Q(0,t^{\star})\vert<\epsilon,
\end{equation}
where $\epsilon=\vert Q(0,t^{\star})\vert>0$. The relation (\ref{eqn:Qcontinuous2}) indicates that choosing $\sigma\in(-\delta,0)$ and placing all the poles at $p_1=p_2=\dots=p_{n}=\sigma$ results in $Q(\sigma,t^{\star})<0$, which means $H(s)$ is not externally positive due to Lemma~\ref{lem:nec&suff}. Choosing $\sigma\in(-\infty,-\delta]$ also leads to a closed-loop system that is not externally positive, because assuming otherwise, the poles could be increased just enough to satisfy
\begin{equation}\label{eqn:p1=p2in-delta,0}
    p_1=p_2=\dots=p_{n}\in(-\delta,0),
\end{equation}
resulting in a closed-loop system that is not externally positive, which is a contradiction to Lemma~\ref{lem:moving_forward}. Therefore, there is no $\sigma<0$ that could make $H(s)$ externally positive with the real-equal poles (\ref{eqn:p1=...=pn}). Recall that Problem~\ref{prob:monotonic_stable} is feasible if and only if it is feasible with a set of real-equal poles (Lemma~\ref{lem:reduction_to_real_equidistanced}). Thus we conclude that Problem~\ref{prob:monotonic_stable} is not feasible with any poles $p\in\mathbb{C}^n$ when $Q(0,t^{\star})<0$.

\noindent\textit{Case 3:} Assume $Q(0,t)\geq 0$ holds for all $t\in(0,+\infty)$ with the equality met at some $t=t^{\star}$. In this case, one may write
\begin{align}\label{eqn:thm_proof1}
    &\frac{d}{d\sigma}Q(\sigma,t^{\star})|_{\sigma=0}=\nonumber\\
    &\sum_{j=0}^{m-1}\binom{n-1}{j}B^{(j+1)}(0){t^{\star}}^{m-j}=\nonumber\\
    &\sum_{j=0}^{m-1}\binom{n-1}{j}(j+1)! b_{n-1-j} {t^{\star}}^{m-j}= \nonumber\\
    &\frac{(n-1)!}{{t^{\star}}^{n-1-m}}\sum_{j=n-m}^{n-1}(n\nonumber-j) \frac{b_j{t^{\star}}^j}{j!}=\nonumber\\
    &\frac{(n-1)!n}{{t^{\star}}^{n-1-m}}\sum_{j=n-m}^{n-1} \frac{b_j {t^{\star}}^j}{j!}-\frac{(n-1)!}{{t^{\star}}^{n-1-m}}\sum_{j=n-m}^{n-1}  \frac{b_j {t^{\star}}^j}{(j-1)!}=\nonumber\\
&\frac{(n-1)!n}{{t^{\star}}^{n-1-m}}\sum_{j=n-m}^{n-1} \frac{b_j {t^{\star}}^j}{j!}
-{t^{\star}}\sum_{j=1}^{m} \binom{n-1}{j} b_{n-j} j! {t^{\star}}^{m-j}=\nonumber\\
&\frac{(n-1)!n}{{t^{\star}}^{n-1-m}}\sum_{j=n-m}^{n-1} \frac{b_j {t^{\star}}^{j}}{j!} -{t^{\star}}\left( Q(0,t^{\star})-b_n {t^{\star}}^{m}\right).
\end{align}
Now, since $Q(0,t^{\star})=0$, it follows from (\ref{eqn:thm_proof1}) that
\begin{align*}
\frac{d}{d\sigma}Q(\sigma,t^{\star})|_{\sigma=0}&=
\frac{(n-1)!n}{{t^{\star}}^{n-1-m}}\sum_{j=n-m}^{n}b_{j} {t^{\star}}^{j}/j! \nonumber\\
&=n{t^{\star}}^{m-n+1}\int_{0}^{t^{\star}} Q(0,\tau)\tau^{n-1-m} d\tau>0,
\end{align*}
where positivity of the right-hand side follows from the assumption $Q(0,t)\geq 0$ for all $t\in(0,+\infty)$. We have shown that $\frac{d}{d\sigma}Q(\sigma,t^{\star}) |_{ \sigma=0}> 0$ always holds in this case, and therefore, there is a small enough perturbation of $\sigma$ that will result in $Q(\sigma,t^{\star})<0$, \emph{i.e.}, there is some $\delta>0$ such that
$$
\sigma\in(-\delta,0) \Rightarrow Q(\sigma,t^{\star})<0,
$$
which results in $H(s)$ not being externally positive due to Lemma~\ref{lem:nec&suff}. Again, similar to the previous case, there is no $\sigma\in(-\infty,0)$ that can make $H(s)$ externally positive with the real-equal poles (\ref{eqn:p1=...=pn}), due to Lemma~\ref{lem:moving_forward}. By invoking Lemma~\ref{lem:reduction_to_real_equidistanced}, we deduce that Problem~\ref{prob:monotonic_stable} is not feasible with any $p\in\mathbb{C}^n$, which proves the necessity of $Q(0,t)>0$ in (\ref{eqn:Q(0)>0}). By combining the three cases above, we conclude that inequality (\ref{eqn:Q(0)>0}) is also necessary for the feasibility of Problem~\ref{prob:monotonic_stable}.
\end{proof}

Theorem~\ref{thm:feasibility} asserts that the feasibility of Problem~\ref{prob:monotonic_stable} is determined by the $m$ roots of polynomial (\ref{eqn:Btild}). These roots depend only on the number of closed-loop poles $n$ and the location of plant zeros $z$ through the numerator coefficients $b$. Since the open-loop poles $p^{\rm o}$ and the static gain $K_{\rm o}$ of the plant do not affect these roots, the feasibility of the monotonic tracking control problem is independent of these parameters.

Instead, the critical parameters affecting feasibility are the chosen controller order $n_c$ through $n=n_c+n_{\rm o}$ and the location of plant zeros $z$ through $b_i$, according to Theorem~\ref{thm:feasibility}. The impacts of these parameters on feasibility are studied separately in the next section. 

\section{Fundamental limitations}\label{sec:fundamental_limits}

\subsection{The location of plant zeros}
We look into how the plant zeros $z$ influence the feasibility of Problem~\ref{prob:monotonic_stable} in a few simple examples.

\begin{example}[Plants with no zero]\label{ex:m=0}
    Assume the plant (\ref{eqn:openloop_transfer_function}) has no zeros, \emph{i.e.}, $m=0$. In this case, we have $b_i=0$ for $i=0,1,\dots,n-1$ and $b_n=K$. Therefore, the polynomial (\ref{eqn:Btild}) is given by
    $$
    \tilde{B}(s)\equiv \tilde{b}_n=K/(n-1)!.
    $$
    Since $\tilde{B}(s)$ does not have any zeros, Problem~\ref{prob:monotonic_stable} is feasible for any controller order, according to Theorem~\ref{thm:feasibility}.
    \noindent 
\hfill$\square$
\end{example}

\begin{example}[Plants with one zero]\label{ex:m=1}
    Assume the plant (\ref{eqn:openloop_transfer_function}) has one zero, \emph{i.e.}, $m=1$. In this case, we have $b_i=0$ for $i=0,1,\dots,n-2$, $b_{n-1}=K$, and $b_{n}=-Kz_1$. Thus, the polynomial (\ref{eqn:Btild}) is given by
    $$
    \tilde{B}(s)=\frac{K}{(n-2)!}s-\frac{Kz_1}{(n-1)!},
    $$
    which has one zero at $s=z_1/(n-1)$. According to Theorem~\ref{thm:feasibility}, Problem~\ref{prob:monotonic_stable} is feasible if and only if this root is negative, \emph{i.e.}, $z_1<0$.
    \noindent 
\hfill$\square$
\end{example}
\begin{example}[Plants with $b_i\geq0$]\label{cor:b>0}
Assume the plant (\ref{eqn:openloop_transfer_function}) has non-negative numerator coefficients as $b_i\geq 0$ for all $i=1,2,\dots,n-1$ and $b_n>0$. In this case, it is obvious from (\ref{eqn:btild=b/(i-1)!}) that all the coefficients of the polynomial $\tilde{B}(s)$ are non-negative. Therefore, $\tilde{B}(s)$ has no positive roots according to Descartes' rule of signs. Since $b_n>0$, it also does not have roots at the origin $s=0$. Therefore, according to Theorem~\ref{thm:feasibility}, Problem~\ref{prob:monotonic_stable} is feasible in this case.
\noindent 
\hfill$\square$
\end{example}

All minimum-phase systems satisfy $b_i\geq 0$ ($i=1,2,\dots,n$). Therefore, Example~\ref{cor:b>0} shows that monotonic tracking is always possible for minimum-phase plants with any controller order. Note that not every system with non-negative numerator coefficients $b_i\geq 0$ is minimum-phase. Inequality $b_i\geq 0$ can be met with zeros located anywhere on the complex plane except the following region~\cite{cow1954}
$$
z\in\lbrace z\vert \, \vert\operatorname{Arg}(z)\vert <\pi/m\rbrace.
$$
The above region shrinks by increasing $m$, but will never include the positive real axis. This result confirms Proposition~\ref{prop:z>=0} in that monotonic tracking is not possible for plants that have zeros on the positive real axis.

Examples~\ref{ex:m=0}-\ref{cor:b>0} specify a few special cases where Problem~\ref{prob:monotonic_stable} is feasible regardless of the controller order. However in the general case, both the location of plant zeros and the controller order play important roles in the feasibility of Problem~\ref{prob:monotonic_stable}, as shown in the next section.

\subsection{Controller order}
The controller order $n_c$ affects the feasibility of Problem~\ref{prob:monotonic_stable} only through the number of closed-loop poles $n=n_c+n_{\rm o}$, where $n_{\rm o}$ is the order of the plant. Hence, for convenience, we investigate the effects of $n$ instead of $n_c$ for that matter. 
The next example shows that the feasibility of Problem~\ref{prob:monotonic_stable} depends on the controller order when the plant has two zeros.



\begin{example}[Plants with two zeros]\label{ex:m=2,different_n}
Assume the plant (\ref{eqn:openloop_transfer_function}) has two zeros, \emph{i.e.}, $m=2$. When the zeros are located in the open left half-plane, all the numerator coefficients $b_i$ are positive, and Problem~\ref{prob:monotonic_stable} is feasible according to Example~\ref{cor:b>0}. When the zeros are real and at least one of them is non-negative, Problem~\ref{prob:monotonic_stable} is infeasible due to Proposition~\ref{prop:z>=0}. Now, assume the zeros
$$
z_{1,2}=u\pm iv
$$
are non-real and located in the closed right half-plane, \emph{i.e.}, $u\geq 0, v\neq 0$. Let $K=1$. The numerator polynomial is $B(s)=s^2-2us+u^2+v^2$, transforming which by (\ref{eqn:Btild}) yields
$$
\tilde{B}(s)=\frac{s^2}{(n-3)!}-\frac{2u s}{(n-2)!}+\frac{u^2+v^2}{(n-1)!}.
$$
The above polynomial does not have any real non-negative roots, if and only if $u/v<\sqrt{n-2}$, or equivalently,
\begin{equation}\label{eqn:Argz>pi/2-..}
    \operatorname{Arg}(z_{1,2})>\frac{\pi}{2}-\arctan(\sqrt{n-2}).
\end{equation}
Inequality (\ref{eqn:Argz>pi/2-..}) clearly shows that increasing the controller order (and hence $n$) expands the region of admissible plant zeros for which Problem~\ref{prob:monotonic_stable} is feasible (see Figure~\ref{fig:cont_order}).
\noindent 
\hfill$\square$
\end{example}

\begin{figure}[!t]
\centerline{\includegraphics[width=0.4\columnwidth]{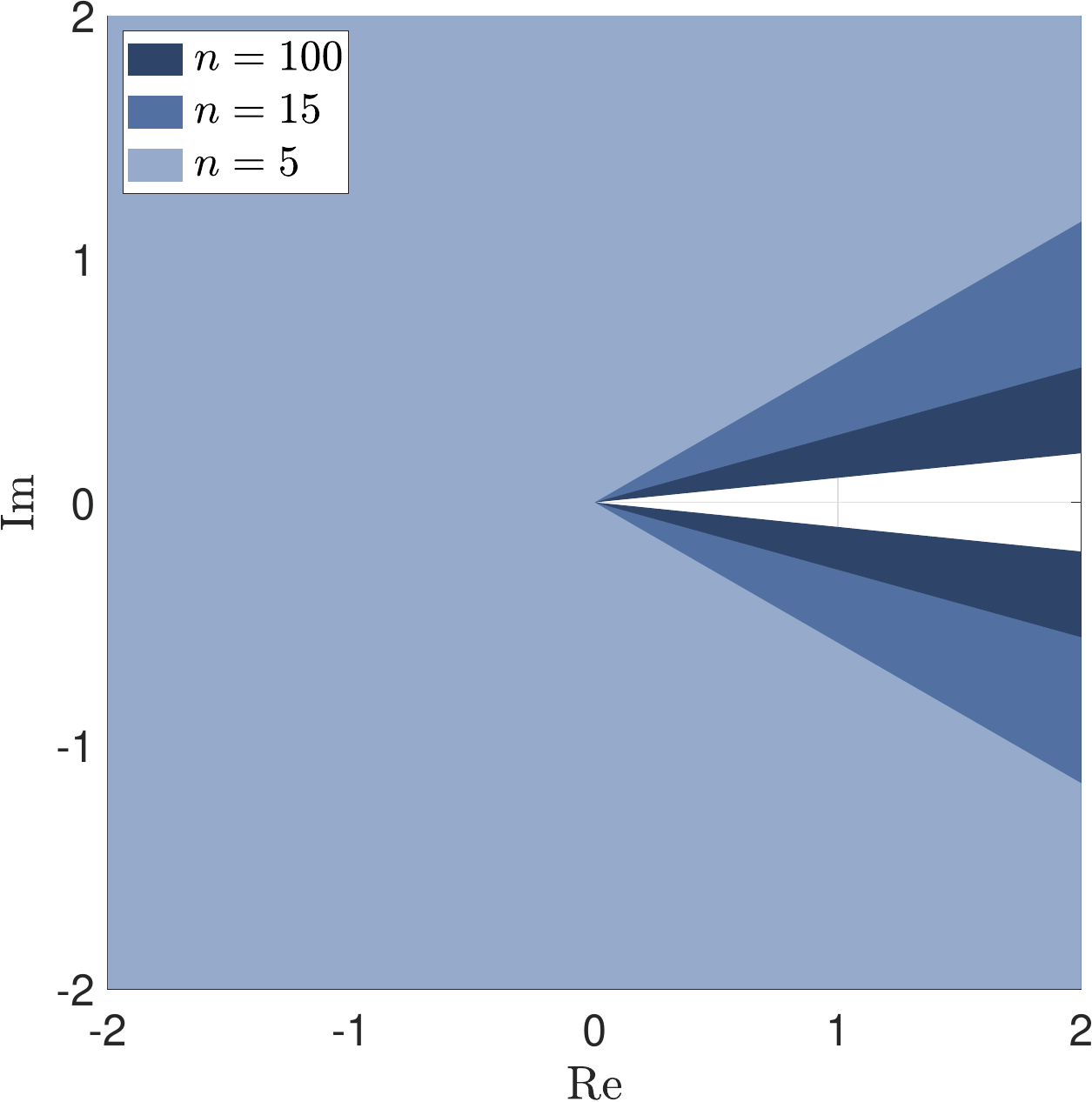}}
\caption{The regions on the complex plane where $z_{1,2}$ satisfy inequality (\ref{eqn:Argz>pi/2-..}) with a fixed $n$. Problem~\ref{prob:monotonic_stable} is feasible using a controller of order $n_c=n-n_{\rm o}$, if and only if the plant zeros are located in these regions, where $n_{\rm o}$ is the order of the plant (fixed). For more information, see Example~\ref{ex:m=2,different_n}.}\label{fig:cont_order}
\end{figure}

Example~\ref{ex:m=2,different_n} shows that Problem~\ref{prob:monotonic_stable} is feasible if and only if the closed-loop system order $n$ is large enough to satisfy (\ref{eqn:Argz>pi/2-..}), when the plant has a pair of complex-conjugate zeros $z_{1,2}$ ($m=2$). The following corollary extends this result to plants with an arbitrary number of real or complex-conjugate zeros ($m\geq 0$) and proves that the necessary condition provided in Proposition~\ref{prop:z>=0} is also sufficient for monotonic tracking.



\begin{corollary}\label{cor:all_feasible}
Problem~\ref{prob:monotonic_stable} is feasible, if and only if the plant~(\ref{eqn:openloop_transfer_function}) does \emph{not} have real non-negative zeros. Furthermore, when Problem~\ref{prob:monotonic_stable} is feasible, it can be solved with the following number of closed-loop poles
\begin{equation}\label{eqn:n^star<all}
    n= m_{\rm Mp}+m_{\rm Np}+\left\lceil\sum\nolimits_{j=1}^{m_{\rm Np}/2} \left({\operatorname{Re}(z_j)}/{\operatorname{Im}(z_j)}\right)^2\right\rceil,
\end{equation}
where $m_{\rm Mp}$ and $m_{\rm Np}$ are the number of plant zeros in the regions $\lbrace z\in\mathbb{C} \vert \operatorname{Re}(z)<0\rbrace$ and $\lbrace z\in\mathbb{C} \vert \operatorname{Re}(z)\geq 0, \operatorname{Im}(z)\neq 0\rbrace$ respectively.
\end{corollary}
\begin{proof}
See Appendix~\ref{sec:proof_of_cor:all_feasible}.
\end{proof}

\subsubsection{Minimum controller order}
Increasing the controller order relaxes the design problem. In other words, if Problem~\ref{prob:monotonic_stable} is feasible for $n=n'$, it is also feasible for all $n>n'$. 
Corollary~\ref{cor:all_feasible} provides an upper bound (\ref{eqn:n^star<all}) on the minimum number of closed-loop poles $n=n^{\star}$ that can solve Problem~\ref{prob:monotonic_stable} when it is feasible. One can find $n^{\star}$ by a simple bisection search over $n$ and using Theorem~\ref{thm:feasibility}: Start with (\ref{eqn:n^star<all}) and decrease $n$ until the polynomial equation $\tilde{B}(s)=0$ has one or more real non-negative roots. The minimum controller order is then given by $n^{\star}_c=n^{\star}-n_{\rm o}$.


\subsection{Closed-loop decay rate}\label{subsec:decayrate}
In addition to Assumptions~\ref{ass:plant} and \ref{ass:controller}, we assume there are no zero-pole cancellations in the closed-loop system (\ref{eqn:transfer_function}) in this section.
Let us define the \emph{decay rate} of the closed-loop system (\ref{eqn:transfer_function}) as its poles abscissa
\begin{equation}\label{eqn:abscissa}
    \sigma(H)=\max_i\lbrace \operatorname{Re}(p_i)\rbrace,
\end{equation}
which determines the speed of the closed-loop system response and its settling time~\cite[p.2]{vidbook}. The next corollary provides the necessary and sufficient conditions for the feasibility of monotonic tracking with a closed-loop decay rate that is not slower than a given value.
\begin{corollary}\label{cor:alfathm}
    Problem~\ref{prob:monotonic_stable} is feasible with decay rate $\sigma(H)<-\alpha$ if and only if the polynomial
\begin{equation}\label{eqn:B'}
    \tilde{B'}(s)=\sum_{i=1}^n \frac{b'_{i} s^{n-i}}{(i-1)!} 
\end{equation}
    does \emph{not} have any real non-negative roots, where
\begin{align}\label{eqn:b'}
B'(s)&=b'_0s^n+b'_1s^{n-1}+\dots+b'_n \nonumber\\
&:=K\prod_{i=1}^m (s-z_i-\alpha).
\end{align}
\end{corollary}
\begin{proof}
    See Appendix~\ref{sec:Proof_of_cor:alfathm}.
\end{proof}
Theorem~\ref{thm:feasibility} is the special case of Corollary~\ref{cor:alfathm} with $\alpha=0$. As Lemma~\ref{lem:moving_forward} suggests, if we have one solution to Problem~\ref{prob:monotonic_stable} we can always find another solution with a larger (slower) decay rate, by moving the poles forward. However the converse is not true; There is often a limit on how fast the closed-loop system can be designed while ensuring a monotonic tracking. The fastest decay rate achievable in Problem~\ref{prob:monotonic_stable} can be defined as
\begin{align} \label{eqn:optsigma}
    \begin{array}[c]{rlll}
    \sigma^{\star}(z,n):=&\underset{p\in\mathbb{C}^n}{\text{inf}} & \sigma(H)\\
    &\mbox{subject to} & \operatorname{Re}(p_i)<0, & i=1,2,\dots,n\\
    & & H(s) \textnormal{ is externally positive}. &
    \end{array}
\end{align}
Note that $\sigma^{\star}(z,n)$ may be unbounded even when Problem~\ref{prob:monotonic_stable} is feasible. This is evident for plants without zeros ($m=0$). For these plants, the closed-loop system (\ref{eqn:transfer_function}) is stable and externally positive regardless of where the closed-loop poles are located in $p\in(-\infty,0)^{n}$. Therefore, an arbitrarily small decay rate can be obtained for the closed-loop system with any controller order, by decreasing the poles abscissa (\ref{eqn:abscissa}). However, when the plant $H_{\rm o}(s)$ has real negative zeros, the closed-loop system decay rate cannot be decreased indefinitely. In this case, the fastest decay rate is lower bounded by~\cite{widder1934}
\begin{equation}\label{eqn:ineq_decay}
    \sigma^{\star}(z,n)\geq \max_{z_i\in\mathbb{R}}\lbrace z_i \rbrace.
\end{equation}


\subsubsection{Fastest decay rate over all controller orders}
Increasing the controller order relaxes the design problem. This means that a smaller (faster) decay rate can be obtained by increasing the controller order. Therefore, for a fixed set of plant zeros $z$, the function $\sigma^{\star}(z,n)$ is \emph{decreasing} with respect to $n$ (and the controller order $n_c$).

The next corollary shows that by increasing the controller order, it is possible to meet the equality in (\ref{eqn:ineq_decay}). This corollary also shows that, when the plant $H_{\rm o}(s)$ does \emph{not} have negative zeros, the closed-loop decay rate (\ref{eqn:optsigma}) can be made arbitrarily small by increasing the controller order.


\begin{corollary}\label{cor:decaynozero}
If the plant (\ref{eqn:openloop_transfer_function}) has real negative zeros, the fastest decay rate (\ref{eqn:optsigma}) of the closed-loop system (\ref{eqn:transfer_function}) satisfies
\begin{equation}\label{eqn:sigstar>z}
\min_n\sigma^{\star}(z,n)= \max_{z_i\in\mathbb{R}}\lbrace z_i \rbrace.
\end{equation}
Otherwsie, if the plant does \emph{not} have any negative zeros, then
$\lim_{n\to +\infty}\sigma^{\star}(z,n)=-\infty$.
\end{corollary}
\begin{proof}
    See Appendix~\ref{sec:Proof_of_cor:decaynozero}.
\end{proof}

Corollary~\ref{cor:decaynozero} provides the fastest decay rate over \emph{all} controller orders. For plants with no zero ($m=0$) and plants with one zero ($m=1$), this decay rate can be realized with any controller order. However, this is generally not true when the plant has two or more zeros ($m\geq 2$), as demonstrated by the following example.

\begin{example}[Plants with two zeros]\label{ex:m=2decay}
    Consider again a plant with two zeros $m=2$ as in Example~\ref{ex:m=2,different_n}. When both zeros are real and negative, one can choose any number $n\geq 3$ of closed-loop poles in the range
    \begin{equation}\label{eqn:exm=2:rangeofp}
        p_i\in(\max_{z_i\in\mathbb{R}}\lbrace z_i \rbrace,0), \quad i=1,2,\dots,n
    \end{equation}
    to solve Problem~\ref{prob:monotonic_stable} with the minimum decay rate $\sigma(H)=\max_{z_i\in\mathbb{R}}\lbrace z_i \rbrace$. To show this, we note that the closed-loop transfer function can be written as $H(s)=H'(s)G(s)$ where
\begin{equation}\label{eqn:H'exm=2}
H'(s)=\frac{K(s-z_1)(s-z_2)}{(s-p_1)(s-p_2)}, 
\end{equation}
    and $G(s)=1/\prod_{i=3}^{n}(s-p_i)$. The transfer function (\ref{eqn:H'exm=2}) is externally positive according to Proposition~\ref{prop:ex_pos_n=2}, and the transfer function $G(s)$ is externally positive as it is a series connection of first-order externally positive systems, according to Proposition~\ref{prop:ex_pos_n=1}. Therefore, $H(s)=H'(s)G(s)$ is externally positive. As this result holds throughout the range (\ref{eqn:exm=2:rangeofp}), the closed-loop system is stable and externally positive even if the closed-loop poles are arbitrarily close to $\max_{z_i\in\mathbb{R}}\lbrace z_i \rbrace$. Therefore, the minimum decay rate can be realized.
    
\noindent Next, we conisder the case where the plant zeros are complex conjugates $z_1,z_2\in\mathbb{C}\backslash\mathbb{R}$. Since the plant does not have any real zeros in this case, the decay rate of the closed-loop system can be arbitrarily small (Corollary~\ref{cor:decaynozero}). However, it is not clear whether $\sigma(H)=-\infty$ can be realized using a \emph{finite} controller order. We use Corollary~\ref{cor:alfathm} to show this is not possible: Using Corollary~\ref{cor:alfathm} and following a similar process as in Example~\ref{ex:m=2,different_n} indicates that a decay rate of $\sigma(H)<-\alpha$ is possible if and only if
\begin{equation}\label{eqn:Argz,alpha,teta}
    \operatorname{Arg}(z_{1,2}+\alpha)>\frac{\pi}{2}-\arctan(\sqrt{n-2}):=\theta(n).
\end{equation} 
As $\alpha\to+\infty$, the left-hand side tends to zero in (\ref{eqn:Argz,alpha,teta}). Hence, the above inequality may only be satisfied when $n\to+\infty$. Therefore, a decay rate of $\sigma(H)=-\infty$ is only achievable by an \emph{infinite} controller order.  Condition~(\ref{eqn:Argz,alpha,teta}) is illustrated in Figure~\ref{fig:limitations}. 
\noindent 
\hfill$\square$
\end{example}

\begin{figure*}[t]
     \centering
     \begin{subfigure}[b]{0.32\textwidth}
        \centering
        \includegraphics[width=1\linewidth]{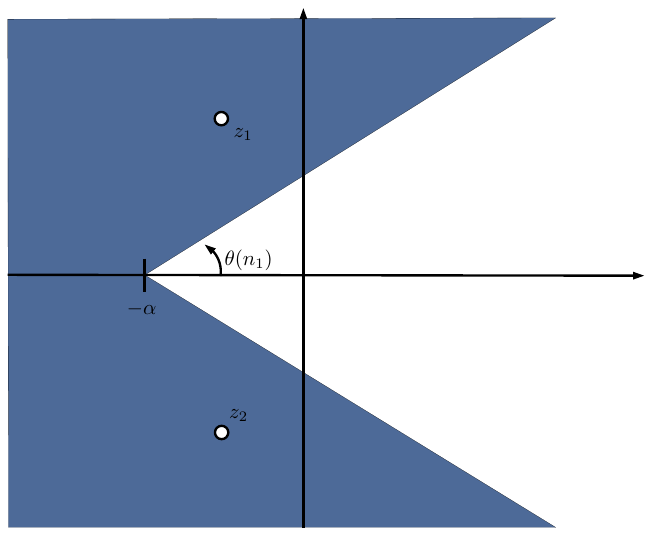}
	    \caption{The shaded region specifies the feasible locations for the plant zeros, the slopes of the straight lines are determined by the controller order, and the point where they intercept the real axis determines the fastest decay rate achievable for the closed-loop system while maintaining a monotonic step response.}
         \label{fig:lim1}
     \end{subfigure}
    \begin{subfigure}[b]{0.32\textwidth}
        \centering
        \includegraphics[width=1\linewidth]{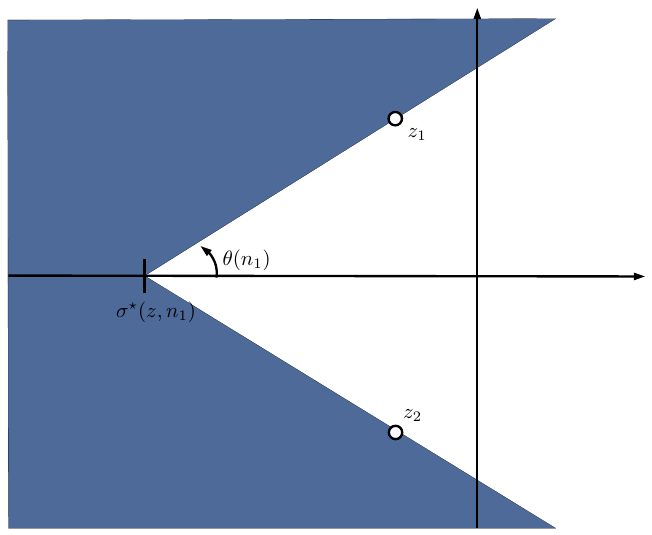}
	    \caption{With a fixed controller order (fixed slope), demanding a faster decay rate (increasing $\alpha$) shifts the shaded region horizontally to the left. In the extreme case where the straight lines cross the plant zeros, the point $-\alpha$ on the real axis specifies the fastest decay rate possible for that plant.}
         \label{fig:lim}
     \end{subfigure}
    \begin{subfigure}[b]{0.32\textwidth}
         \centering
        \includegraphics[width=1\linewidth]{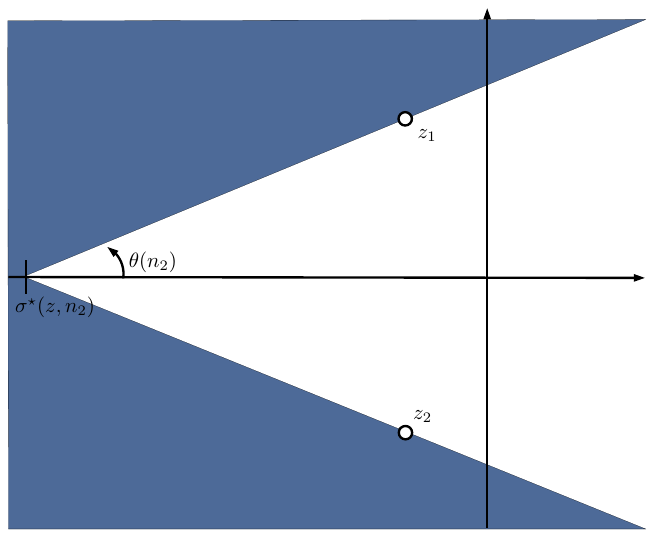}
	    \caption{Both the fastest decay rate $\sigma^{\star}(z,n)$ and the angle $\vert\theta(n)\vert$ are decreasing with $n$ (see also Figure~\ref{fig:cont_order}). Hence, by using $n=n_2$ where $n_2>n_1$, one can achieve a faster decay rate for the closed-loop system while maintaining a monotonic response. This is realized by increasing the controller order.} 
       \label{fig:lim2}
    \end{subfigure}
        \caption{
        The regions on the complex plane where $z_{1,2}$ satisfy the inequality $\operatorname{Arg}(z_{1,2}+\alpha)>\theta(n)$ in (\ref{eqn:Argz,alpha,teta}). Problem~\ref{prob:monotonic_stable} is feasible using a controller of order $n_c=n-n_{\rm o}$ and a closed-loop decay rate $\sigma(H)<-\alpha$, if and only if the plant zeros are located in these regions, where $n_{\rm o}$ is the order of the plant (fixed). For more information, see Example~\ref{ex:m=2decay}.
        }
        \label{fig:limitations}
\end{figure*}

As $\sigma^{\star}(z,n)$ is decreasing with $n$, the minimum decay rate over all controller orders offered in Corollary~\ref{cor:decaynozero} is typically achieved by high-order controllers and, in some cases, such as plants with one pair of complex-conjugate zeros (Example~\ref{ex:m=2decay}), only when $n\to+\infty$. However, when the controller order (and hence the number of closed-loop poles $n$) is \emph{fixed} at a certain value, it may not be possible for the closed-loop system to have the decay rate (\ref{eqn:sigstar>z}).

\subsubsection{Fastest decay rate with a given controller order}
The function $\sigma^{\star}(z,n)$ defined in (\ref{eqn:optsigma}) gives the fastest decay rate achievable for the closed-loop system with a fixed set of zeros $z$ and a fixed number of closed-loop poles $n$. In the special case $m=2$, when the plant has a pair of complex-conjugate zeros, the fastest decay rate $\sigma^{\star}(z,n)$ is given by the following relation
\begin{equation}\label{eqn:simpleformula}
    \sigma^{\star}(z,n)=\operatorname{Re}(z_{1,2})-
\vert\operatorname{Im}(z_{1,2})\vert\sqrt{n-2},
\end{equation}
which is obtained from a slight manipulation of equation (\ref{eqn:Argz,alpha,teta}). The above equation shows that achieving a faster closed-loop response becomes increasingly difficult when the plant zeros are closer to the real axis. This is also shown to be true when $m>2$ in Example~\ref{ex:decayrate}. In the general case $m\geq 0$, there are two possible ways to calculate $\sigma^{\star}(z,n)$.


The first approach relies on Corollary~\ref{cor:alfathm}: Start with $\alpha=0$. If monotonic tracking is feasible with the given controller order, the polynomial (\ref{eqn:Btild}) does not have real non-negative roots. In this case, the minimum decay rate can be found via a bisection search over $\alpha$. We increase $\alpha$ until the polynomial $\tilde{B'}$ in (\ref{eqn:B'}) has one or more non-negative roots. The maximum $\alpha=\alpha_{\rm max}$ found in this way, determines the fastest decay rate as follows
$$
\sigma^{\star}(z,n)=-\alpha_{\rm max}.
$$

An alternative way to compute $\sigma^{\star}(z,n)$ is by solving the optimization problem (\ref{eqn:optsigma}) directly. To simplify this task, we first note in the proof of Lemma~\ref{lem:reduction_to_real_equidistanced} that when Problem~\ref{prob:monotonic_stable} has a feasible solution in $p\in\mathbb{C}^{n}$ with decay rate $\sigma(H)$, it also has a feasible solution with real-equal poles $p_1=p_2=\dots=p_n=\sigma$ (\ref{eqn:p1=...=pn}) and the same decay rate. This means restricting the poles to be real and equal does not compromise the fastest decay rate achievable. Therefore, adding the constraint (\ref{eqn:p1=...=pn}) to the optimization problem (\ref{eqn:optsigma}) does not change its optimal value $\sigma^{\star}(z,n)$. However, it makes the optimization problem significantly simpler as it allows using Lemma~\ref{lem:nec&suff} to verify the external positivity constraint in (\ref{eqn:optsigma}). Therefore, to solve (\ref{eqn:optsigma}), one can simply use a bisection search over the parameter $\sigma\in(-\infty,0)$ and check the non-negativity of the polynomial $Q(\sigma,t)$ in (\ref{eqn:Q(x)}). Since all non-negative univariate polynomials are sums of squares (SOS)~\cite{hilbert}, the non-negativity of $Q(\sigma,t)$ can be easily posed as a semi-definite program (SDP). However, since we only require non-negativity in a semi-infinite range in Lemma~\ref{lem:nec&suff}, enforcing $Q(\sigma,t)$ to be SOS would be conservative. Nevertheless, the Markov-Lucaks' theorem helps to recover an SDP formulation with no conservatism, as shown in the following proposition.

\begin{proposition}[\cite{Lucaks,MDB2009}]\label{prop:Lucaks}
Polynomial $Q(\sigma,t)$ of degree $m$ is non-negative in the range $t\in[0,+\infty)$, if and only if there exist polynomials $y'(t)$ and $y''(t)$ of degrees at most $m/2$ and $(m-1)/2$ respectively, such that
\begin{equation}\label{eqn:Qq1q2}
    Q(\sigma,t-1)={y'}^2(t)+(t-1){y''}^2(t).
\end{equation}
\end{proposition}
\begin{proof}
By using \cite[Lemma~3]{MDB2009}, one can verify the non-negativity of a polynomial in the range $[1,+\infty)$. Applying this result to the shifted polynomial (\ref{eqn:Qq1q2}) is equivalent to the condition $Q(\sigma,t)\geq 0$ for $t\in(0,+\infty)$.
\end{proof}
 
Condition (\ref{eqn:Qq1q2}) can be readily posed as an SDP by equating the corresponding coefficients of the polynomials appearing on the both sides of the equation (\ref{eqn:Qq1q2}), as shown in \cite{MDB2009}.

\begin{example}\label{ex:decayrate}
Consider the following sixth-order transfer function
\begin{figure*}
\begin{equation}\label{eqn:sysEx2}
    H_{\rm o}=\frac{0.094 s^3 +20 s^2 +2.4\times 10^3 s +3.5 \times 10^5}{1.2 \times 10^{-3}s^6 +2.8 s^5 +2\times 10^3 s^4 +3.9 \times 10^{5} s^3 +8.7 \times 10^7 s^2 +6.4 \times 10^9 s +6.4 \times 10^{11}},
\end{equation}
\end{figure*}
which relates the sensed force to the actuator force in a PUMA 560 manipulator~\cite{puma}. This system has its zeros approximately at
$z_{1,2,3}=-187,-16 \pm 141i$. Monotonic tracking using the controller structure (\ref{eqn:two-par-restricted}) is feasible for the plant (\ref{eqn:sysEx2}), because it does not have any non-negative zeros (Corollary~\ref{cor:all_feasible}). However, because of the zero on the real axis, the fastest decay rate achievable is bounded by (\ref{eqn:sigstar>z}). We choose $n_c=n_{\rm o}-1=5$ to comply with Assumption~\ref{ass:controller} which gives $n=n_c+n_{\rm o}=11$. Then we follow the SDP method descried above to obtain $$\sigma^{\star}(z,11)=-187,$$
which means the fastest decay rate (\ref{eqn:sigstar>z}) over all controller orders is already realizable with $n_c=5$. To also investigate the effect of the complex-conjugate zeros on the decay rate, we changed a few mechanical parameters of the manipulator in \cite{puma} to shift the plant zeros to $z_{1,2,3}=-187,-16 \pm 40i$ in which only the complex-conjugate zeros have changed. By using the same number of closed-loop poles $n=11$ and the SDP method described in Section~\ref{subsec:decayrate}, the fastest decay rate was calculated to be much larger, \emph{i.e.} $\sigma^{\star}(z,n)=-132$ in this case. This indicates that shifting the complex-conjugate zeros closer to the real axis can limit the control system's ability to achieve a fast response.
\noindent 
\hfill$\square$
\end{example}

\section{Conclusion}\label{sec:conclusion}
We revisited the classical problem of designing linear output-feedback controllers for tracking reference set points monotonically, without overshoots or undershoots. We showed that whenever this problem is feasible, it can be solved using the control structure shown in Figure~\ref{fig:control_scheme}. 
We obtained the necessary and sufficient conditions for the feasibility of monotonic tracking with a decay rate that is faster than a given value (Corollary~\ref{cor:alfathm}). These conditions are based on the plant zeros, the number of plant poles, and the controller order, as the only decisive parameters. We obtained a method to compute the minimum controller order and the fastest closed-loop decay rate achievable in a monotonic tracking system based on these results.

Our results show that monotonic tracking is possible for a plant if and only if it does not have \emph{positive} zeros (Corollary~\ref{cor:all_feasible}). Furthermore, when monotonic tracking is feasible, the closed-loop system can be designed with an arbitrarily fast decay rate, if and only if the plant does not have \emph{negative} zeros (Corollary~\ref{cor:decaynozero}). Roughly speaking, monotonic tracking is more difficult when the plant zeros have larger real parts and smaller absolute imaginary parts, \emph{i.e.}, closer to the positive real axis. In this case, one may settle for slower closed-loop system responses and/or higher controller orders to achieve monotonic tracking. This fact is demonstrated algebraically in (\ref{eqn:simpleformula}) and geometrically in Figures~\ref{fig:cont_order} and \ref{fig:limitations} for plants with only one pair of complex-conjugate zeros.

\section*{References}

\bibliography{main.bib}
\bibliographystyle{ieeetr}

\appendix

\subsection{Proof of Proposition~\ref{prop:PIRnecK>0}}\label{sec:Proof_of_prop:PIRnecK>0}
There are two possible cases:\\
\noindent\textit{Case $m=n$:} the closed-loop impulse response $h(t)=\mathcal{L}^{-1}\lbrace H(s)\rbrace$ satisfies
\begin{equation}\label{eqn:propnecproof1}
    h(t)=K\delta(t)+\mathcal{L}^{-1}\lbrace H(s)-K\rbrace
\end{equation}
where $\mathcal{L}^{-1}\lbrace H(s)-K\rbrace$ is bounded in a neighbourhood of $t=0$ and $\delta(t)$ is the Dirac delta function. If $H(s)$ is externally positive, $h(t)$ is non-negative, and therefore, inequality (\ref{eqn:Kcl>0}) holds due to (\ref{eqn:propnecproof1}).\\
\noindent\textit{Case $n>m$:} we have
\begin{equation}\label{eqn:K>0h(0)>0}
\left\lbrace
\begin{array}{ll}
{d^{k-1} }h(t)/{ds^{k-1}}|_{t=0^+}=0, &  k=1,2,\dots,n-m-1 \\
{d^{k-1} }h(t)/{ds^{k-1}}|_{t=0^+}=K, &  k=n-m
\end{array}
\right.
\end{equation}
according to the initial value theorem. Again, in this case, inequality (\ref{eqn:Kcl>0}) is necessary for the non-negativity of $h(t)$.
\noindent 
\hfill$\square$

\subsection{Proof of Lemma~\ref{lem:nec&suff}}\label{sec:Proof_of_lem:nec&suff}
It is deduced from (\ref{eqn:K>0h(0)>0}) that $h(t)=\mathcal{L}^{-1}\lbrace H(s)\rbrace$ is non-negative in a neighbourhood of $t=0$, if and only if $K> 0$. Let $t\neq 0$. The inverse Laplace transform of the transfer function $H(s)=B(s)/(s-\sigma)^{n}$ satisfies \cite[p.231]{bateman1954tables}
\begin{align*}
    &h(t) \exp(-\sigma t)=\\
    &b_1+\left(b_2+\binom{n-1}{1}b_1\sigma\right)t\\
    &+\left(b_3+\binom{n-2}{1}b_2\sigma+\binom{n-1}{2}b_1\sigma^2\right)t^2/2!\\
    &+\dots+\left(b_{n}+b_{n-1}\sigma+\dots+b_1\sigma^{n-1}\right)t^{n-1}/(n-1)!
\end{align*}
which, after a little manipulation, can be written as
$$
h(t)=
\frac{\exp(\sigma t)}{(n-1)!}\sum_{j=0}^{n-1}\binom{n-1}{j}B^{(j)}(\sigma)t^{n-1-j}
$$
as the numerator polynomial $B(s)$ is of order $m$, $B^{(j)}(s)\equiv 0$ holds for $j>m$. Therefore, we arrive at
$$
h(t)=
\exp(\sigma t)\frac{t^{n-1-m}}{(n-1)!}Q(\sigma,t)
$$
which proves $h(t)\geq 0$ holds if and only if $Q(\sigma,t)\geq 0$ for all $t\in(0,+\infty)$.
\noindent 
\hfill$\square$

\subsection{Proof of Lemma~\ref{lem:moving_forward}}\label{sec:Proof_of_lem:moving_forward}
Transfer function (\ref{eqn:H'_1}) can be written as $H'(s)=H(s)G(s)$ where
$$
G(s)=\frac{\prod_{i=1}^{n} (s-p_i)}{\prod_{i=1}^{n} (s-p_i-\delta_i)}
$$
is a series connection of the first order systems $(s-p_i)/(s-p_i-\delta_i)$ that are all externally positive according to Proposition~\ref{prop:ex_pos_n=1}. Since $H(s)$ is also externally positive by assumption, its series connection with $G(s)$, \emph{i.e.} the transfer function $H'(s)=H(s)G(s)$ is externally positive.
\noindent 
\hfill$\square$

\subsection{Proof of Lemma~\ref{lem:reduction_to_real_equidistanced}}\label{sec:Proof_of_lem:reduction_to_real_equidistanced}
\textit{if part:} Obvious.

\textit{only if part:}
If Problem~\ref{prob:monotonic_stable} is feasible, then there are $n$ poles $p'_i$ ($i=1,2,\dots,n$) in the open left half plane such that
\begin{equation}\label{eqn:H'}
H'(s)={B(s)}/{\prod_{i=1}^{n} (s-p'_i)}
\end{equation}
is externally positive. Without loss of generality, we assume that the first $2k$ of these poles are complex-conjugates, \emph{i.e.}
$$
p'_{2i-1}=\bar{p}'_{2i},\quad i=1,2,\dots,k
$$
Now define a new set of poles $p''\in (-\infty,0)^{n}$ as
\begin{equation}\label{eqn:p''}
    p''_i=\operatorname{Re}(p'_i), \quad i=1,2,\dots,n
\end{equation}
Replacing the closed-loop system poles $p'$ by $p''$ gives the transfer function $H''(s)=H'(s)G(s)$ where
$$
G(s)=\prod_{i=1}^k\frac{(s-p'_{2i-1})(s-p'_{2i})}{(s-p''_{2i-1})(s-p''_{2i})}
$$
is a series connection of second-order systems that are externally positive according to Proposition~\ref{prop:ex_pos_n=2}. Therefore, $G(s)$ is externally positive, which proves $H''(s)$ is externally positive. Now, define a new set of poles $p\in(-\infty,0)^{n}$ as
$$
p_i=p''_i+\delta_i
$$
where $\delta_i=\max_j\lbrace p''_j \rbrace-p''_i\geq 0$. According to Lemma~\ref{lem:moving_forward}, the transfer function (\ref{eqn:transfer_function}) is externally positive with poles $p_i$ all equal to $\sigma=\max_j\lbrace p''_j \rbrace$. Hence, when Problem~\ref{prob:monotonic_stable} is feasible, it can always be solved by a set of real-equal poles.
\noindent 
\hfill$\square$


\subsection{Proof of Corollary~\ref{cor:all_feasible}}\label{sec:proof_of_cor:all_feasible}
Consider factorizing the closed-loop system (\ref{eqn:transfer_function}) based on its minimum-phase and non-minimum-phase subsystems as $H(s)=KH_{\rm Mp}(s)H_{\rm Np}(s)$, where
\begin{align*}
H_{\rm Mp}(s)&=\frac{\prod_{i=1}^{m_{\rm Mp}}(s-z_{i})}{A_0(s)},\\
H_{\rm Np}(s)&=\prod_{i=1}^{m_{\rm Np}/2}\frac{(s-z_{m_{\rm Mp}+2i-1})(s-z_{m_{\rm Mp}+2i})}{A_{i}(s)}
\end{align*}
where it is assumed that the first $m_{\rm Mp}$ zeros are minimum-phase and
\begin{equation}\label{eqn:complex_m1}
    z_{m_{\rm Mp}+2i-1}=\bar{z}_{m_{\rm Mp}+2i},\quad i=1,2,\dots,m_{\rm Np}/2
\end{equation}
are the non-minimum-phase complex-conjugate zeros. Since $z_1,z_2,\dots,z_{m_{\rm Mp}}$ are all located in the left half-plane, there is a stable polynomial $A_0(s)$ that makes $H_{\rm Mp}(s)$ externally positive and the order of $A_0(s)$ can be as low as $n_0=m_{\rm Mp}$ (Example~\ref{cor:b>0}). There is a stable polynomial $A_i(s)$ of order $n_i$ in each factor of $H_{\rm Np}(s)$ that makes that factor externally positive where $\sqrt{n_i-2}>\operatorname{Re}(z_{m_{\rm Mp}+2i})/\operatorname{Im}(z_{m_{\rm Mp}+2i})$ according to (\ref{eqn:Argz>pi/2-..}). Since the series connection of externally positive systems is itself externally positive, we conclude that there is a stable polynomial
$$A(s)=A_0(s)\prod_{i=1}^{m_{\rm Np}/2} A_i(s)$$ of order $n=n_0+\sum_{i=1}^{m_{\rm Np}/2} n_i$ which makes the closed-loop system (\ref{eqn:transfer_function}) externally positive. Therefore, Problem~\ref{prob:monotonic_stable} is feasible with (\ref{eqn:n^star<all}).
\noindent 
\hfill$\square$


\subsection{Proof of Corollary~\ref{cor:alfathm}}\label{sec:Proof_of_cor:alfathm}
First, we prove that Problem~\ref{prob:monotonic_stable} is feasible with decay rate $\sigma(H)<-\alpha$ if and only if it is feasible with the shifted zeros $z_i+\alpha$ where $i=1,2,\dots,m$. Then we use this result and Theorem~\ref{thm:feasibility} to prove this corollary. 

\textit{if part:} Let Problem~\ref{prob:monotonic_stable} be feasible with the shifted zeros $z_i+\alpha$ where $i=1,2,\dots,m$. In this case, there are $n$ closed-loop poles $p_i$ with $\operatorname{Re}(p_i)<0$ such that the transfer function
$$
H'(s)=\frac{\prod_{i=1}^m (s-z_i-\alpha)}{\prod_{i=1}^n (s-p_i)}
$$
is externally positive. Let $h'(t)$ be the impulse response of $H'(s)$. Then the impulse response of the transfer function $H''(s)=H'(s+\alpha)$ is given by $h'(t)\exp(-\alpha t)$. Therefore, $H'(s)$ is externally positive if and only if
$$
H''(s)=H'(s+\alpha)=\frac{\prod_{i=1}^m (s-z_i)}{\prod_{i=1}^n (s-p_i+\alpha)}
$$
is externally positive. Hence, there are $n$ closed-loop poles $p_i$ with $\operatorname{Re}(p_i)<0$ such that $H''(s)$ is externally positive. This condition is equivalent to the existence of $n$ closed-loop poles $p''_i=p_i-\alpha$ with $\operatorname{Re}(p''_i)<-\alpha$ such that the following transfer function is externally positive
$$
H(s)=K\frac{\prod_{i=1}^m (s-z_i)}{\prod_{i=1}^n (s-p''_i)}.
$$

\textit{only if part:} Obvious, as all the above steps are reversible.

Therefore, if we first shift the plant zeros $z_i$ as $z_i+\alpha$ and then apply Theorem~\ref{thm:feasibility}, we obtain the necessary and sufficient conditions of feasibility in Problem~\ref{prob:monotonic_stable} with the decay rate $\sigma(H)<-\alpha$.

\noindent 
\hfill$\square$

\subsection{Proof of Corollary~\ref{cor:decaynozero}}\label{sec:Proof_of_cor:decaynozero}
In the proof of Corollary~\ref{cor:alfathm}, we showed that Problem~\ref{prob:monotonic_stable} has a solution with the closed-loop decay rate $\sigma(H)<-\alpha$ if and only if Problem~\ref{prob:monotonic_stable} is feasible with the shifted zeros 
\begin{equation}\label{shiftz}
z_i+\alpha,\quad i=1,2,\dots,m.
\end{equation}
We use the above result to prove this corollary by considering two separate cases.

\noindent\textit{Plants with real zeros:}
According to Corollary~\ref{cor:all_feasible}, when the plant has real negative zeros, Problem~\ref{prob:monotonic_stable} is feasible with the shifted zeros (\ref{shiftz}) if and only if $z_i+\alpha<0$ holds for all $i\in 1,2,\dots,m$ such that $z_i\in\mathbb{R}$. This condition is equivalent to $\alpha<-\max_{z_i\in\mathbb{R}}\lbrace z_i \rbrace$. Therefore, the fastest decay rate is given by (\ref{eqn:sigstar>z}).

\noindent\textit{Plants without real zeros:}
Let $m_1$ and $m_2$ be the number of plant zeros in the regions $\lbrace z\in\mathbb{C} \vert \operatorname{Re}(z)<-\alpha\rbrace$ and $\lbrace z\in\mathbb{C} \vert \operatorname{Re}(z)\geq -\alpha, \operatorname{Im}(z)\neq 0\rbrace$ respectively. According to Corollary~\ref{cor:all_feasible}, when the zeros are shifted as (\ref{shiftz}), Problem~\ref{prob:monotonic_stable} is feasible  for all $\alpha\in \left[0,\min_{z_i\in\mathbb{R}}\lbrace -z_i\rbrace \right)$ and all $n$ satisfying
\begin{equation}\label{eqn:n^star<all_shifted}
    n\geq 1+m_1+m_2+\sum_{j=1}^{m_2/2}\left\lfloor \left(\frac{\operatorname{Re}(z_j)+\alpha}{\operatorname{Im}(z_j)}\right)^2\right\rfloor
\end{equation}
Hence, the decay rate $\sigma(H)<-\alpha$ is achieved in Problem~\ref{prob:monotonic_stable} by $n$ suitable closed-loop poles $p_i$ ($i=1,2,\dots,n$) in the open left-half plane, where $n$ satisfies (\ref{eqn:n^star<all_shifted}). As $\alpha\to+\infty$, the bound on the right-hand side of (\ref{eqn:n^star<all_shifted}) tends to infinity. Therefore, an arbitrarily fast decay rate is possible if $n$ is allowed to grow arbitrarily large, \emph{i.e.}, $\lim_{n\to +\infty}\sigma^{\star}(z,n)=-\infty$.

\noindent 
\hfill$\square$



\end{document}